\documentclass[12pt]{amsart}
\usepackage{amssymb}
\usepackage{amsthm}
\usepackage{amsmath}
\usepackage{bm}
\usepackage{enumerate}

\theoremstyle{plain}
\newtheorem{theorem}{Theorem}[section]
\newtheorem{lemma}[theorem]{Lemma}
\newtheorem{proposition}[theorem]{Proposition}
\newtheorem{corollary}[theorem]{Corollary}

\theoremstyle{remark}
\newtheorem{Def}[theorem]{Definition}
\newtheorem{example}[theorem]{Example}

\newtheorem{Rem}[theorem]{Remark}

\numberwithin{equation}{section}

\newcommand{\CC}{{\mathbb C}}

\newcommand{\RR}{{\mathbb R}}

\newcommand\Span{{\mathop{\mbox{\rm span }}}}

\newcommand{\T}{\mathcal{T}}

\newcommand{\Mtz}{{M_\Lambda^1}}

\newcommand{\supp}{\mathop{\rm supp}}

\usepackage{color}
\newcommand{\tcr}[1]{{#1}}

\def\<{\langle}
\def\>{\rangle}

\newcommand{\eps}{\epsilon}

\newcommand{\Pwr}[1]{x^{\lambda_{#1}}}
\newcommand{\xln}{\Pwr{n}}

\title{Embedding Theorems for M\"untz spaces}
\author[I. Chalendar]{Isabelle Chalendar}
\author[E. Fricain]{Emmanuel Fricain}
\author[D. Timotin]{Dan Timotin}
\address[I. Chalendar and E. Fricain]{Universit\'e de Lyon;
Universit\'e Lyon 1; INSA de Lyon; Ecole Centrale de Lyon;
CNRS, UMR5208, Institut Camille Jordan; 43 bld. du 11 novembre 1918,
F-69622 Villeurbanne Cedex, France}
\email{chalenda@math.univ-lyon1.fr, fricain@math.univ-lyon1.fr}
\address[D. Timotin]{Institute of Mathematics of the 
Romanian Academy, PO Box 1-764, Bucharest 014700, Romania}
\email{Dan.Timotin@imar.ro}
\thanks{The authors were partially supported by the ANR project FRAB}


\begin{document}
\bibliographystyle{plain}
\begin{abstract}
We discuss boundedness and compactness properties of the embedding
$\Mtz\subset L^1(\mu)$, where $\Mtz$ is the closure of the monomials
$x^{\lambda_n}$ in $L^1([0,1])$ and $\mu$ is a finite positive Borel measure on the
interval $[0,1]$.
 In particular, we introduce a class of "sublinear" measures and
 provide a rather
 complete solution of the embedding problem for the class of quasilacunary
sequences $\Lambda$. Finally, we show how one can recapture  some
 of Al Alam's results on boundedness and essential norm of  weighted
composition operators from $\Mtz$ to $L^1([0,1])$.
\end{abstract}
\maketitle

\section{Introduction}

A classical result due to M\"untz says that, if $0=\lambda_0<\lambda_1<\dots<\lambda_n<\dots$ is an increasing sequence of nonnegative \tcr{real} numbers, then the linear span of $x^{\lambda_n}$ is dense in $C([0,1])$ if and only if $\sum_{n=1}^\infty \frac{1}{\lambda_n}=\infty$. When $\sum_{n=1}^\infty \frac{1}{\lambda_n}<\infty$, the \tcr{closed linear span of} the monomials $x^{\lambda_n}$ in different Banach spaces that contain them is usually not equal to the whole space; it is therefore interesting to study the new spaces thus obtained. \tcr{In particular, if $1\leq p<+\infty$ and $\sum_{n=1}^\infty \frac{1}{\lambda_n}<\infty$, then $M_\Lambda^p\varsubsetneq L^p([0,1])$, where $M_\Lambda^p$ is the closed linear span of the monomial $x^{\lambda_n}$, $n\geq 0$, in  $L^p([0,1])$.} The literature concerning this class of spaces of functions defined on
$[0,1]$, called \emph{M\"untz spaces}, is not very extensive. We may
refer principally to the two monographs \cite{be95} and \cite{gl05} and the references within,
as well as to the recent papers \cite{aa09,sp08}.  

The starting point of our research is formed by some recent results of
Ihab Al Alam, either published in \cite{aa09} or contained in his
thesis \cite{aathesis}. They deal with properties of weighted composition operators on M\"untz spaces, and it is noted therein that the properties of these operators are connected to embedding of the spaces into Lebesgue spaces.

In the present paper 
we have pursued this line of approach systematically. More precisely, we discuss boundedness and compactness properties of the embedding $\Mtz\subset L^1(\mu)$, where $\mu$ is a finite \tcr{positive} Borel measure on the interval $[0,1]$. In general the embedding properties are critically dependent on the nature of the sequence $\Lambda=(\lambda_n)$. 

The plan of the paper is the following. Section~2 contains preliminaries as well as general results concerning boundedness of the embedding, while Section~3 discusses compactness of the embedding. In Section~4 we introduce an important class of measures that we call sublinear, and which bear a certain resemblance to Carleson measures defined in the unit disc or half-plane. These allow in Section~5 a rather complete solution of the embedding problem for the class of quasilacunary sequences $\Lambda$. Section~6 discusses through some examples the problems that may appear in the general case, while Section~7 investigates the important sequence $\lambda_n=n^2$. Finally, in Section~8 we show how one can recapture in this context some of Al Alam's results on weighted composition operators.

\section{Embedding measures}

The basic reference for M\"untz spaces is \cite{gl05}; occasionally we will use also some results from \cite{be95}.

We denote by $m$ the Lebesgue measure on $[0,1]$ and by  $\Lambda=
(\lambda_n)_{n\ge 1}$  an increasing sequence of positive real numbers
with $\sum_{n\geq 1} \frac{1}{\lambda_n}<\infty$. 
The norm in $L^p(m)$ will be denoted simply by $\|\cdot\|_p$ (for $1\le p\le\infty$).
The closed linear span in $L^1(m)$ of the functions $x^{\lambda_n}$ is the M\"untz space $\Mtz$. The functions in $\Mtz$ are continuous in $[0,1)$ and real analytic in~$(0,1)$.

Note that sometimes M\"untz spaces are defined by including the value $\lambda_0=0$ in the family of monomials; but the statements are simpler if we start with $\lambda_1>0$. This is a matter of convenience only: all boundedness and compactness results below remain true if we add the one-dimensional space formed by the constants.

\begin{Def}
\tcr{A positive measure $\mu$ on $[0,1]$ is called \emph{$\Lambda$-embedding} if there is a constant $C>0$ such that \begin{equation}\label{eq:def-plongement}
\|p\|_{L^1(\mu)}\le C\|p\|_1, 
\end{equation}
for all polynomials $p$ in $M^1_\Lambda$.}
\end{Def}

It is immediate (by applying the condition to the functions $x^{\lambda_n}$) that if $\mu$ is $\Lambda$-embedding, then $\mu(\{1\})=0$, so we will suppose this condition satisfied for all measures appearing in this paper. If $0<\eps<1$, then the interval $[1-\eps, 1]$ will be denoted by $J_\eps$.

The next lemma is a useful technical tool.

\begin{lemma}\label{le:rho}
Suppose $\rho:\RR_+\to\RR_+$ is an increasing, $C^1$ function with
$\rho(0)=0$ such that $\mu(J_\eps)\le\rho(\eps)$ for all $\eps\in (0,1]$. Then for any continuous, positive, increasing function $g$  we have 
\[
\int_{[0,1]}g\, d\mu \le \int_0^1 g(x)\rho'(1-x)\, dx.
\]
\end{lemma}

\begin{proof}
If $F(x)=\mu([0,x))$, then integration by parts yields
\[
\int_{[0,1]}g\, d\mu =\int_0^1 g(x)\, dF(x)= g(1)F(1)-\int_0^1 F(x)\, dg(x)
\]
(remember $F(0)=0$). Since $\mu(J_\eps)\le\rho(\eps)$ for all $\eps\in
(0,1]$ (and $\mu(\{1\})=0$), it follows that 
 $F(1)-F(x)\le \rho(1-x)$, or $-F(x)\le \rho(1-x)-F(1)$. Plugging this into the previous equation, integrating again by parts, and using $\rho(0)=0$, we obtain
\[
\begin{split}
\int_{[0,1]}g\, d\mu &\le g(1)F(1)-F(1)(g(1)-g(0))+ \int_0^1 \rho(1-x) \, dg(x)\\
&=F(1)g(0) +\int_0^1 \rho(1-x)\, dg(x)\\
&= F(1)g(0) -\rho(1)g(0)-\int_0^1 g(x)\, d(\rho(1-x))\\
&=g(0) (F(1)-\rho(1)) + \int_0^1 g(x)\rho'(1-x)\, dx.
\end{split}
\]
The proof is finished by noting that $F(1)=\mu([0,1))\le \rho(1)$.
\end{proof}

We intend to investigate necessary and sufficient conditions for a measure $\mu$ to be $\Lambda$-embedding. These conditions depend in general on the sequence $\Lambda$; also, in most cases there is a gap between necessity and sufficiency. The starting point for our approach is the following result from \cite[p.185, E.8.a]{be95}.

\begin{lemma}\label{le:BE}
For all $\eps\in (0,1]$ there exists a constant $c_\eps>0$ such that for any function $f\in \Mtz$ we have
\begin{equation}\label{eq:BE}
\sup_{0\le t\le 1-\eps}|f(t)|\le  c_\eps \int_{1-\eps}^1
|f(x)|dx.\end{equation}
In particular, the supremum in the left-hand side is majorized by
$c_\eps \|f\|_1$.  
\end{lemma}

\tcr{Formally the inequality \eqref{eq:BE} is stated in \cite{be95} only for polynomials in $M_\Lambda^1$, but a standard argument shows that it can be extended to any function $f\in \Mtz$.}

Lemma~\ref{le:BE} has some immediate consequences for the embedding problem.

\begin{corollary}\label{co:compactsupport}
(i) If, for some $\eps>0$,  $\supp \mu\subset [0,1-\eps)$, then $\mu$ is $\Lambda$-embedding for any $\Lambda$, and 
\[
\|f\|_{L^1(\mu)}\le  c_\eps \|\mu\| \|f\|_1
\]
for all $f\in\Mtz$.

(ii) More generally, if for some $\eps>0$ the restriction of $\mu$ to the interval $[1-\eps, 1]$ is absolutely continuous with respect to \tcr{$m_{|[1-\eps,1]}$}, with essentially bounded density, then $\mu$ is $\Lambda$-embedding for any $\Lambda$.
\end{corollary}

\begin{Rem}\label{Rem:rem-plongement}
\tcr{If $\mu$ is $\Lambda$-embedding, then $M_\Lambda^1\subset L^1(\mu)$ and $\|f\|_{L^1(\mu)}\leq c \|f\|_1$ for all $f\in M_\Lambda^1$. Indeed if $f\in M_\Lambda^1$, then there exists a sequence $(p_n)_n$ of polynomials in $M_\Lambda^1$ such that $\|f-p_n\|_1\to 0$, $n\to +\infty$. By~\eqref{eq:def-plongement} $(p_n)_n$ is a Cauchy sequence in $L^1(\mu)$, whence it converges to a function $g$ in $L^1(\mu)$. In particular, there exists a subsequence $(p_{n_k})_k$ which converges almost everywhere (with respect to $\mu$) to $g$. But according to Lemma~\ref{le:BE},  $(p_n)_n$ tends to $f$ uniformly on every compact of $[0,1)$, so $g(t)=f(t)$ for almost every $t\in [0,1)$ with respect to $\mu$; since $\mu(\{1\})=0$, we have $g=f$ in $L^1(\mu)$. Therefore $M_\Lambda^1\subset L^1(\mu)$. Moreover, since $(p_n)_n$ tends to $f$ in $L^1(\mu)$ and also in $L^1(m)$,~\eqref{eq:def-plongement} implies that $\|f\|_{L^1(\mu)}\leq c \|f\|_1$, which proves the claim.}

\tcr{Using standard arguments based on the closed graph theorem, it is easy to see that it is sufficient to have the set inclusion $M_\Lambda^1\subset L^1(\mu)$ in order to obtain that $\mu$ is $\Lambda$-embedding.} For a $\Lambda$-embedding $\mu$ we  denote by $\iota_\mu$ the embedding operator $\iota_\mu: M_\Lambda^1\subset L^1(\mu)$.
\end{Rem}


To obtain a more general \tcr{sufficient condition}, note
that  the smallest constant that can
appear in the right-hand side of (\ref{eq:BE}) is a positive, decreasing function of $\eps$, and thus admits
decreasing and continuous majorants. In the sequel we fix such a
majorant, denoted by $c(\eps)$.      
Then $\kappa(t):=c(1-t)$ is positive, increasing \tcr{and continuous} on $[0,1)$. Using~\eqref{eq:BE} for any monomial $x^{\lambda_i}$, we see that $\kappa(t)\to\infty$ for $t\to 1$, and the order of increase is at least~$(1-t)^{-1}$.

\begin{theorem}\label{th:generalsufficient}
If $\kappa\in L^1(\mu)$, then $\mu$ is $\Lambda$-embedding and $\|\iota_\mu\|\le\|\kappa\|_{L^1(\mu)}$.
\end{theorem} 

\begin{proof} For $k,m\ge 1$,
define the set $E_{m,k}=\{x\in [0,1): \frac{k-1}{2^m}< \kappa(x)\le\frac{k}{2^m}\} $. Since $\kappa$ is increasing and continuous, we have $E_{m,k}=(a^{(m)}_{k-1},a^{(m)}_k]$ for some $a^{(m)}_k\in [0,1)$. If the functions $\kappa_m$ are defined by
\[
\kappa_m(x)=\sum_{k=0}^\infty \kappa(a^{(m)}_k) \chi_{E_{m,k}}(x),
\]
then $\kappa_m$ is a decreasing sequence of functions tending everywhere to $\kappa$. By the monotone convergence theorem, it follows from the hypothesis that 
\begin{equation}\label{eq:gensuf1}
\sum_{k=1}^\infty \mu ((a^{(m)}_{k-1},a^{(m)}_k]) \kappa(a^{(m)}_k)\to \|\kappa\|_{L^1(\mu)}<\infty.
\end{equation}

On the other hand, from~\eqref{eq:BE} it follows that, \tcr{for any $f\in M_\Lambda^1$}, 
\[
\sup_{a^{(m)}_{k-1}<x\le a^{(m)}_k} |f(x)|\le \kappa(a^{(m)}_k)\|f\|_1,
\]
whence (taking into account that $f(0)=0$ and $\mu(\{1\})=0$)
\[
\begin{split}
\int_{[0,1]}|f(x)|\, d\mu(x)& =
\sum_{k=1}^\infty \int_{(a^{(m)}_{k-1},a^{(m)}_k]}| f(x)|\, d\mu(x)\\
&\le \sum_{k=1}^\infty \left(\sup_{a^{(m)}_{k-1}<t\le a^{(m)}_k} |f(x)|\right) \mu((a^{(m)}_{k-1},a^{(m)}_k])\\
&\le \sum_{k=1}^\infty  \kappa(a^{(m)}_k)\|f\|_1 \mu((a^{(m)}_{k-1},a^{(m)}_k]).
\end{split}
\]
Letting now $m\to\infty$ and using~\eqref{eq:gensuf1}, it follows that $\mu$ is $\Lambda$-embedding and $\|\iota_\mu\|\le\|\kappa\|_{L^1(\mu)}$.
\end{proof}

Since $\kappa$ is \tcr{continuous,} increasing and $\kappa(t)\to +\infty$ for $t\to 1$, the condition $\kappa\in L^1(\mu)$ is related to the asymptotics of $\mu$ in~1. The next corollary gives a sufficient condition on $\mu$ expressed in terms of $\mu(J_\eps)$.

\begin{corollary}\label{co:mu(Jeps)}
Suppose  $\rho:\RR_+\to\RR_+$ is an increasing, $C^1$ function with $\rho(0)=0$ such that $\int_0^1 \kappa(x)\rho'(1-x)\, dx<\infty$. If
$\mu(J_\eps)\le\rho(\eps)$ for all $\eps\in (0,1]$, then $\mu$ is $\Lambda$-embedding.
\end{corollary}

\begin{proof}
The proof follows by applying Lemma~\ref{le:rho} to the case $g=\kappa$.
\end{proof}

In general $\kappa$ has not an explicit formula in terms of the sequence $\Lambda$.  We would like to obtain more treatable formulas; this will be done below for some special classes of $\Lambda$.

\section{Compactness and essential norm}

A related problem is that of the compactness of  the embedding $\iota_\mu$. Here, the starting point is a result stated as Lemma 4.2.5 in \cite{aathesis}; we will sketch a proof for completeness.

\begin{lemma}\label{le:ihab}
If $(f_m)_m\subset\Mtz$, $\|f_m\|_1\le1$ for all $m$, then there exists a subsequence $(f_{m_k})_k$ which converges uniformly on every compact subset of $[0,1)$.
\end{lemma}

\begin{proof}
For any $\eps\in (0,1]$ the restrictions of $f_m$ to $[0,1-\eps]$ are
uniformly bounded by Lemma~\ref{le:BE}. To prove the result, it is
enough to consider only functions $f_m$ in the linear span of
$x^{\lambda_n}$ with $\lambda_n\ge 1$; it follows then from a
Bernstein-type inequality proven in \cite[p.178, E.3.d]{be95} that the
restrictions of $f'_m$ to $[0,1-2\eps]$ are uniformly bounded. By the
Arzela--Ascoli theorem, \tcr{the sequence ${f_n}_{|[0,1-2\eps]}$ contains a subsequence which is uniformly convergent on $[0,1-2\eps]$}. Applying this to $\eps=1/N$ for all positive integer $N$ and using a diagonal procedure, one obtains a subsequence $f_{m_j}$ that converges uniformly on all compact of \tcr{$[0,1)$} to a function $f$.  It is then easy to show that $f\in\Mtz$.
\end{proof}

We may now improve Corollary~\ref{co:compactsupport}.

\begin{proposition}\label{pr:bicompact}
If $\supp \mu\subset [0,1-\eps]$, then $\iota_\mu$  is compact.
\end{proposition}

\begin{proof}
The proof is an immediate consequence of Lemma~\ref{le:ihab}: if we take a sequence $(f_n)_n$ in the unit ball of $\Mtz$, then the subsequence  obtained therein converges uniformly in $[0,1-\eps]$, and therefore also in $L^1(\mu)$.
\end{proof}

The following notation will be useful: if $\mu$ is a positive measure
on $[0,1]$, we will denote $\mu_m$ the measure equal to 
$\mu$ on $[0, 1-\frac{1}{m})$ and $0$ elsewhere, and $\mu'_m=\mu-\mu_m$ (so $\mu'_m$ is the measure that coincides with $\mu$ on 
$J_{\frac{1}{m}}$ and is $0$ elsewhere).

Next comes a general abstract compactness result, related to Corollary~\ref{co:mu(Jeps)}.

\begin{proposition}\label{pr:abstractcompactness}
Let $\Mtz$ be a M\"untz space, and suppose that $\rho:\RR_+\to \RR_+$ is an increasing continuous function, with $\rho(0)=0$, such that any measure $\mu$ with $\mu(J_\eps)\le C\rho(\eps)$ is $\Lambda$-embedding, with embedding constant at most $C$. Then, for any measure $\mu$ that satisfies $\lim_{\eps\to0}\frac{\mu(J_\eps)}{\rho(\eps)}= 0$ the embedding $\iota_\mu$ is compact.
\end{proposition}

\begin{proof}
Suppose that $\iota_{\mu_m}$ are the embeddings $\Mtz\subset L^1(\mu_m)$. We may regard $L^1(\mu_m)$ as a subspace of $L^1(\mu)$, and the operators $\iota_{\mu_m}$ as taking values in $L^1(\mu)$.
The hypothesis implies that $\mu'_m=\mu- \mu_m$ is
$\Lambda$-embedding, with embedding constant tending to $0$. Therefore
$\|\iota_\mu-\iota_{\mu_m}\|\to 0$. Since  each $\iota_{\mu_m}$ is
compact, by Proposition~\ref{pr:bicompact}, the result follows.
\end{proof}

For a general bounded embedding, one can use the measures $\mu'_m$  to obtain a formula for the essential norm of $\iota_\mu$. We start with an abstract simple lemma that will be used also in Section~\ref{se:compo}.

\begin{lemma}\label{le:abstractessential}
Let $X$ be a Banach space, $(E,\nu)$ a \tcr{measurable} space, $T:X\to L^1(\nu)$ a bounded operator, 
$(E_m)_m$  a decreasing sequence of measurable subsets of $E$ such that $\nu(\bigcap_m E_m )=0$. \tcr{Suppose that $(I-P_m)$ is compact for all $m$, where $P_m$ denotes the natural projection of $L^1(\nu)$ onto $L^1(E_m,\nu)$}. Then \tcr{the essential norm of $T$ is given by}
\[
\|T\|_e=\lim_{m\to\infty} \|P_mT\|. 
\]
\end{lemma}

\begin{proof} Note that the sequence $(\|P_mT\|)_m$ is decreasing, so the limit exists.
 Since $T-P_mT$ is compact for each $m$, it is obvious that 
$\|T\|_e\le \lim_{m\to\infty} \|P_mT\|$. To prove the reverse
inequality, let $\eps>0$ and $K:X\to L^1(\nu)$ be a compact
operator. Take a sequence $(x_m)_m\subset  X$ such that $\|x_m\|=1$ and $\|P_mT x_m\|_{\tcr{L^1(E_m,\nu)}}\ge \|P_mT\|-\eps$. Then $(Kx_m)_m$ contains a convergent subsequence, say $Kx_{m_j}\to g\in L^1(\nu)$. We have
\[
 \|(T-K)x_{m_j}\|_{\tcr{L^1(\nu)}}\ge \|Tx_{m_j} -g\|_{\tcr{L^1(\nu)}}-\|Kx_{m_j}-g\|_{\tcr{L^1(\nu)}},
\]
whence
 \begin{equation}\label{eq:3.0}
 \limsup_j \|(T-K)x_{m_j}\|_{\tcr{L^1(\nu)}}\ge\limsup_j\|Tx_{m_j} -g\|_{\tcr{L^1(\nu)}}.
\end{equation}
Since $g\in L^1(\nu)$ and $\nu(\bigcap_m E_m )=0$, there exists a
positive integer $N$ such that 
$ \int_{E_m}|g|\,d\nu \le \eps$
for all $m\ge N$. Then, if $m_j\ge N$, we have
\[
 \begin{split}
  \|Tx_{m_j} -g\|_{\tcr{L^1(\nu)}}& =\int_E |Tx_{m_j} -g|\, d\nu 
\ge \int_{E_{m_j}} |Tx_{m_j} -g|\, d\nu
\ge \int_{E_{m_j}} |Tx_{m_j}| d\nu-\eps\\
&=\|P_{m_j}Tx_{m_j}\|_{\tcr{L^1(E_{m_j},\nu)}}-\eps\ge \|P_{m_j}T\|-2\eps.
 \end{split}
\]
From~\eqref{eq:3.0} it follows then that
\[
\|T-K\| \ge \lim_m \|P_{m}T\|-2\eps.
\]
Since this is true for any compact $K$, we have 
\[
 \|T\|_e\ge \lim_m \|P_{m}T\|-2\eps.
\]
Letting $\eps\to 0 $ yields the desired inequality.
\end{proof}

\begin{theorem}\label{th:essentialnorm}
 Let $\Mtz$ be a M\"untz space, and suppose that $\mu$ is an embedding measure. Then
\begin{equation}\label{eq:essentialnorm}
 \|\iota_\mu\|_e = \lim_{n\to\infty}  \|\iota_{\mu'_n}\|.
\end{equation}
\end{theorem}

\begin{proof} 
We may apply Lemma~\ref{le:abstractessential} to the case $X=\Mtz$, $(E,\nu)=([0,1], \mu)$, $T=\iota_\mu$, and $E_m=J_{\frac{1}{m}}$. The compactness condition on \tcr{$(I-P_m)$} follows from Proposition~\ref{pr:bicompact}. 
\end{proof}

The \tcr{formula} given in Theorem~\ref{th:essentialnorm} can be made explicit in particular cases. We state an example that will be used in Section~\ref{se:compo}.

\begin{corollary}\label{co:essential-abs-cont-limit}
Let $\Mtz$ be a M\"untz space, and suppose there exists $\delta>0$ such that \tcr{$d\mu_{| J_\delta}= h \,dm_{|J_\delta}$} for some bounded measurable function $h$ with $\lim_{t\to 1} h(t)=a$. Then $\iota_\mu$ is bounded and $\|\iota_\mu\|_e= a$.
\end{corollary}

\begin{proof} The boundedness of $\iota_\mu$ is a consequence of Corollary~\ref{co:compactsupport}~(ii). To obtain the essential norm, according to Theorem~\ref{th:essentialnorm}, we have to prove that $\lim_{m\to\infty}  \|\iota_{\mu'_m}\|=a$. First, for any $\eps>0$, if $m$ is sufficiently large, then $d\mu'_m\le (a+\eps) dm$ on $J_\frac{1}{m}$, which implies $\|\iota_{\mu'_m}\|\le a+\eps$; therefore
\[
 \lim_{m\to\infty}  \|\iota_{\mu'_m}\|\le a.
\]

On the other side, fix again $\eps>0$. If $m$ is sufficiently large, then $d\mu'_m\ge (a-\eps) dm$ on $J_\frac{1}{m}$. Take  the function $g_n(x)=(\lambda_n+1)x^{\lambda_n}$. We have $g_n\in\Mtz$, $\|g_n\|_\Mtz=1$, while
\[
 \|\iota_{\mu'_m} g_n \| 
=\int_{J_\frac{1}{m}}g_n\, d\mu 
\ge (a-\eps)\int_{1-\frac{1}{m}}^1 g_n(x)\, dx
=(a-\eps)\Big[1-\big(1-\frac{1}{m}\big)^{\lambda_n+1}\Big],
\]
and the last quantity tends to $a-\eps$ for $n\to\infty$. Therefore $ \|\iota_{\mu'_m}  \|\ge a-\eps$ for all $\eps>0$. Letting now $\eps\to 0$ yields the desired reverse inequality.
\end{proof}

\section{Sublinear measures}

We start with the following simple observation.

\begin{lemma}\label{le:the function}
If $\mu$ is $\Lambda$-embedding, then for any $n$ we have $\mu(J_{\frac{1}{\lambda_n}})\le C \frac{1}{\lambda_n}$.
In particular, $\liminf_{\eps\to 0}\frac{\mu(J_\eps)}{\eps}<\infty$. 
\end{lemma}

\begin{proof}
Since
$\lim_{n\to\infty}\left(1-\frac{1}{\lambda_n}\right)^{\lambda_n}=\frac{1}{e}$,
there exists a positive integer $N$ such that, for all $n\geq N$ and for
all $x\in \left[1-\frac{1}{\lambda_n},1\right]$, we have
$x^{\lambda_n}\geq \frac{1}{3}$. It follows that for all $n\geq N$ 
\[\frac{1}{3}\mu(J_{1/\lambda_n})\leq
\int_{J_{1/\lambda_n}}x^{\lambda_n}d\mu\leq \int_0^1
x^{\lambda_n}d\mu \leq \|\iota_\mu\|\int_0^1
x^{\lambda_n}dx=\frac{\|\iota_\mu\|}{\lambda_n +1}.\]
Therefore, for all $n\geq N$, we have    
\[\mu(J_{1/\lambda_n})\leq\frac{3\|\iota_\mu\|}{\lambda_n}.\]
It is now obvious that there exists $C>0$ such that 
$\mu(J_{1/\lambda_n})\leq\frac{C}{\lambda_n}$ for all $n$.  
\end{proof}

This suggests the following definition.

\begin{Def}
A measure $\mu$ is called \emph{sublinear} 
if there is a constant $C>0$ such that for any $0<\eps<1$ we have
$\mu(J_\eps)\le C \eps$; the smallest such $C$ will be denoted by
$\|\mu\|_S$. The measure $\mu$ is called
 \emph{vanishing sublinear} if $\lim_{\eps\to 0}\frac{\mu(J_\eps)}{\eps}=0$.

\end{Def}

We can then obtain a necessary condition for $\Lambda$-embedding for a class of sequences $\Lambda$. 

\begin{proposition}\label{pr:sublinear} Suppose that there exists $M>0$ such that $\frac{\lambda_{n+1}}{\lambda_n}\le M$. Then any $\Lambda$-embedding measure is sublinear.
\end{proposition}

\begin{proof}
\tcr{Obviously it is sufficient to check the condition of sublinearity for $\eps$ sufficiently small. Since $1/\lambda_n$ is a decreasing sequence tending to $0$, we may assume that for some $n$ we have
$\frac{1}{\lambda_{n+1}}\le \eps\le \frac{1}{\lambda_n}$.} Then, applying Lemma~\ref{le:the function},
\[
\mu(J_\eps)\le \mu(J_{\frac{1}{\lambda_n}} )\le \frac{C}{\lambda_n}
\le \frac{CM}{\lambda_{n+1}}\le CM\eps. \qedhere
\]
\end{proof}

It is interesting that sublinearity is also a sufficient condition for $\Lambda$-embedding for the class of quasilacunary sequences $\Lambda$. This will be proved in the next section; now we continue with some elementary facts about sublinear measures.

The next lemma is a consequence of Lemma~\ref{le:rho}, in case $\rho(x)=x$.

\begin{lemma}\label{le:second}
Suppose that $\|\mu\|_S<\infty$. If $g$ is continuous, positive, and increasing, we have
\begin{equation}\label{eq:11}
\int_{[0,1]}g\, d\mu\le \|\mu\|_S \int_{[0,1]}g\, dm.
\end{equation}
\end{lemma}

\begin{corollary}\label{co:first} (i) If $\mu$ is sublinear, then 
 \[
\sup_{\lambda}\|(\lambda+1)\Pwr{}\|_{L^1(\mu)}\le \|\mu\|_S.
\]
(ii) If  $\mu$ is vanishing sublinear, then 
 \[
\lim_{\lambda\to\infty}\|\lambda\Pwr{}\|_{L^1(\mu)}=0.
\]
\end{corollary}

\begin{proof}
For (i), we apply  Lemma~\ref{le:second} to the case $g=\Pwr{}$.

For (ii), fix $\eps>0$. By the definition of vanishing sublinear, there is $\delta>0$ such that $\mu(J_{\delta'})\le \eps\delta'$ for all $\delta'\le\delta$. If $\lambda$ is large enough, then $\lambda x^\lambda\le \eps$ for all $x\le 1-\delta$, and we have
\[
\int_{[0,1]}\lambda\Pwr{}\,d\mu(x)= \int_{[0,1-\delta)}\lambda\Pwr{}\,d\mu(x)+  \int_{[1-\delta,1)}\lambda\Pwr{}\,d\mu(x)
\le \eps\|\mu\|+  \int_{[1-\delta,1)}\lambda\Pwr{}\,d\mu(x).
\] 
\tcr{Let $\mu_\delta$ be the measure which is equal to $\mu$ on
  $J_\delta$ and is $0$ elsewhere. Then applying again
  Lemma~\ref{le:second} to the measure 
$\mu_{\delta}$ and to the function $g(x)=\lambda\Pwr{}$, we get that} the second term is also bounded by~$\eps$.
\end{proof}

\section{Quasilacunary sequences}

\tcr{A sequence $\Lambda$ is \emph{lacunary} if for some $q>1$ we have $\lambda_{n+1}/\lambda_n\geq q$, $n\geq 1$. More generally a sequence $\Lambda$ is \emph{quasilacunary} if for some increasing
sequence $(n_k)$ of integers and some $q>1$ we have $\lambda_{n_{k+1}}/\lambda_{n_k}\ge q$ and $N:=\sup_k(n_{k+1}-n_k)<\infty$. (The first condition of 
 quasilacunarity is sometimes stated as} $\lambda_{n_k+1}/\lambda_{n_k}\ge q$; it is easy to see that the two are equivalent, of course with a different $q$.) The sequence $\Lambda$ will be fixed in this section.
We need a few results from \cite{gl05}.

\begin{lemma}[{\cite[Corollary 6.1.3]{gl05}}]\label{le:majorize_f} 
There exists a constant $d>0$ (depending only on $\Lambda$) such that,
if $f(x)=\sum_{n=1}^m \alpha_n x^{\lambda_n}$, then, for all $x\in[0,1]$,
\[
|f(x)|\le d   x^{\lambda_1} \|f\|_\infty. 
\]
\end{lemma}

\begin{lemma}[{\cite[Proposition~8.2.2]{gl05}}]\label{le:bernstein} 
There is a constant $K>0$ (depending only on $\Lambda$) such that, if $f(x)=\sum_{n=1}^m \alpha_n x^{\lambda_n}$, then
\[
\|f'\|_\infty\le K \left(\sum_{n=1}^m \lambda_n \right)\|f\|_\infty.
\]
\end{lemma}

\begin{lemma}[{\cite[Theorem~9.3.3]{gl05}}]\label{le:quasil1}
If $\Lambda$ is quasilacunary, $F_k=\Span\{x^{\lambda_{n_k+1}},\ldots, x^{\lambda_{n_{k+1}}}\}$, then there is $d_1>0$ such that for any sequence of functions $f_k\in F_k$ we have
\[
d_1 \sum_k \|f_k\|_1 \le \|\sum_k f_k\|_1 \le  \sum_k \|f_k\|_1.
\]
\end{lemma}

\tcr{In particular, if $\Lambda$ is lacunary, then there is $d_1>0$ such that 
\begin{equation}\label{eq:lacunary}
d_1\sum_n \frac{|a_n|}{\lambda_n}\leq \|p\|_1\leq \sum_n \frac{|a_n|}{\lambda_n},
\end{equation}
for every polynomial $p(x)=\sum_n a_n x^{\lambda_n}$ in $M_\Lambda^1$.}

Let us also note that, in the proof of \cite[Theorem 9.3.3]{gl05}, it is shown that any quasilacunary sequence may be enlarged to one that is still quasilacunary and satisfies $\lambda_{n+1}/\lambda_n\le q^2$ for all $n$. We will suppose this true in the sequel; it follows then that we can assume
\begin{equation}\label{eq:ratio_lambdas}
q\le \frac{\lambda_{n_{k+1}}}{\lambda_{n_k+1}}\le q^{2(N-1)}.
\end{equation}

Finally, we will use also the following elementary lemma.

\begin{lemma}\label{le:elementary}
If $f:[0,1]\to \RR$ is a nonconstant differentiable function,
  then
\[
\|f\|_1\ge \min\left\{\frac{\|f\|_\infty^2}{2\|f'\|_\infty},
\frac{\|f\|_\infty}{4}\right\}.
\]
\end{lemma}
\begin{proof}
Let $x_0\in [0,1]$ such that $|f(x_0)|=M$, where $M:=\|f\|_\infty>0$.
Replacing $f$ by $-f$ we may assume that $f(x_0)=M$. Obviously 
one of the intervals $[x_0-1/2,x_0]$ or  $[x_0,x_0+1/2]$ is in $[0,1]$.
 Suppose that the first interval is in $[0,1]$. \\
If $\frac{\|f\|_\infty}{4}\leq \frac{\|f\|_\infty^2}{2\|f'\|_\infty}$, i.e.
 if $\|f'\|_\infty\leq 2M$, for all $x\in [x_0-1/2,x_0]$,
\begin{equation*}\label{eq:triangle}
f(x)\geq 2M (x- x_0)+M\geq 0.
\end{equation*}
It follows that $\|f \|_1\geq \int_{x_0-1/2}^{x_0}f(x)dx\geq \frac{M}{4}$.\\
If $\|f'\|_\infty >2M$ then $x_0-\frac{M}{\|f'\|_\infty}\in  [x_0-1/2,x_0]$
 and  for all $x\in [x_0-M/\|f'\|_\infty,x_0]$,
\begin{equation*}\label{eq:triangle2} 
f(x)\geq \|f'\|_\infty (x- x_0)+M\geq 0.
\end{equation*}
It follows that $\|f \|_1\geq \int_{x_0-M/\|f'\|_\infty}^{x_0}f(x)dx\geq
\frac{M^2}{2\|f'\|_\infty}$.\\
The proof in the case where $[x_0,x_0+1/2]$ is in $[0,1]$ follows along
 the same lines.
\end{proof}

We are ready now for the promised extension.

\begin{theorem}\label{th:Sarquasilac2}
If $\Lambda$ is quasilacunary, then any sublinear measure $\mu$ is $\Lambda$-embedding.
\end{theorem}

\begin{proof}
By Lemma~\ref{le:quasil1}, it is enough to show that there is a constant $C>0$ such that for any $k$ and any $f\in F_k$ we have $\|f\|_{L^1(\mu)}\le C\|f\|_1$. Let us then fix $k$ and $f\in F_k$. Applying Lemma~\ref{le:majorize_f} and Corollary~\ref{co:first}, we have
\begin{equation}\label{eq:estimate}
\begin{split}
\|f\|_{L^1(\mu)}&=
\int_0^1 |f(x)|\, d\mu(x)
\le  d\|f\|_\infty  \int_0^1 x^{\lambda_{n_k+1}}\, d\mu(x)\\
&\le    \frac{ d\|\mu\|_S \|f\|_\infty}{\lambda_{n_k+1}} .
\end{split}
\end{equation}

We apply now Lemma~\ref{le:elementary}; there are two possible cases. If $\|f\|_\infty\le 4\|f\|_1$, then it follows immediately from~\eqref{eq:estimate} that
\[
\|f\|_{L^1(\mu)}\le \frac{4 d\|\mu\|_S \|f\|_1}{\lambda_{n_k+1}}.
\]
We obtain thus the desired inequality and the theorem is proved in this case.

If the minimum in Lemma~\ref{le:elementary} is given by the other term, we use Lemma~\ref{le:bernstein}, which yields
\begin{equation}\label{eq:bernstein_here}
\|f'\|_\infty\le K \left(\sum_{n=n_k+1}^{n_{k+1}} \lambda_n \right)\|f\|_\infty \le KN\lambda_{n_{k+1}} \|f\|_\infty.
\end{equation}
Then, by Lemma~\ref{le:elementary},
\[
\|f\|_\infty^2\le 2\|f\|_1\|f'\|_\infty\le 2KN \lambda_{n_{k+1}} \|f\|_1\|f\|_\infty.
\] 
We cancel $\|f\|_\infty$ from both sides and  plug the resulting inequality  into~\eqref{eq:estimate}. With the aid of~\eqref{eq:ratio_lambdas}, one obtains
\[
\|f\|_{L^1(\mu)} \le  \frac{2KNd \lambda_{n_{k+1}} \|\mu\|_S \|f\|_\infty }{\lambda_{n_k+1}} \le 2KNd q^{2(N-1)} \|\mu\|_S \|f\|_1,
\]
which finishes the proof.
\end{proof}

Combining Theorem~\ref{th:Sarquasilac2} with Proposition~\ref{pr:sublinear}, we obtain a class of $\Lambda$ for which sublinearity is a necessary and sufficient condition for $\Lambda$-embedding.

\begin{corollary}\label{co:bibounded}
If, for some  increasing sequence $(n_k)$ of integers with $\sup_k(n_{k+1}-n_k)<\infty$ we have
\[
1<\inf_k \frac{\lambda_{n_{k+1}}}{\lambda_{n_k}}\le \sup_k \frac{\lambda_{n_{k+1}}}{\lambda_{n_k}} <\infty,
\]
then a measure $\mu$ is $\Lambda$-embedding if and only if it is sublinear.
\end{corollary}

The next corollary follows from Proposition~\ref{pr:abstractcompactness} \tcr{and Theorem~\ref{th:Sarquasilac2}}.

\begin{corollary}\label{co:compactsublinear}
If $\Lambda$ is quasilacunary, then for any vanishing sublinear measure the embedding $\iota_\mu$ is compact.
\end{corollary}

\section{Some examples}

In general the property of being $\Lambda$-embedding depends on the sequence $\Lambda$, as shown by the next example.

\begin{example}\label{ex:twolambdas}
We intend to construct a measure $\mu$ and two lacunary sequences $\Lambda=(\lambda_n)_n$ and $\Lambda'=(\lambda'_n)_n$ such that $\mu$ is $\Lambda$-embedding but not $\Lambda'$-embedding.

For $0<a<1$, consider the function  $\lambda\mapsto \lambda a^{\lambda}$. It attains its maximum in $x_a=-\frac{1}{\ln a}$, and the maximum value is between $x_a/2$ and $x_a$; also, the limit for $\lambda\to\infty$ is 0. Since, for $a$ close to 1, $-\ln a$ is of the same order as $1-a$, it follows that $x_a$ is of the same order as $(1-a)^{-1}$.

Take then $\mu=\sum_k c_k \delta_{a_k}$, where $\delta_x$ defines the Dirac measure in $x$, $a_k\uparrow 1$ and $c_k>0$, $\sum_k c_k<\infty$. The values $c_k, a_k, \lambda_k$ will be defined recurrently. We start with $c_0=\lambda_0=1$ and $a_0=0$.  Suppose then that they have been defined for $k\ge n$. We choose then $\lambda_{n+1}$ large enough such that $\lambda_{n+1} a_k^{\lambda_{n+1}}\le 1$ for all $k\le n$. Then we take $a_{n+1}$ such that 

(i) $1-a_{n+1}\le \frac{1-a_n}{2}$;

(ii) $(n+1)(1-a_{n+1})\le \frac{1}{\lambda_{n+1}}$.\\
Finally, we define $c_{n+1}=(n+1)(1-a_{n+1})$.

Note first that (i) implies $1-a_n\le 2^{-n}$ and therefore
 $c:=\sum_k c_k\le \sum_k k 2^{-k}<\infty$. Moreover, using (i), we
 have also $c_{n+1}\leq (2/3)c_n$ for all $n\geq 3$. Now, since 
 $\int x^\lambda \, d\mu=\sum_k c_k a_k^{\lambda}$, it follows that,
 for all $n\geq 3$
 \[
\begin{split}
\int \lambda_n   x^{\lambda_n} \, d\mu &=\sum_k c_k \lambda_n a_k^{\lambda_n}
=\sum_{k<n} c_k \lambda_n a_k^{\lambda_n}+\sum_{k\ge n} c_k \lambda_n a_k^{\lambda_n}\\
&\le \sum_{k< n} c_k +\sum_{k\ge n} c_k \lambda_n 
\le c+ \sum_{k\geq n}(2/3)^{n-k}c_n\lambda_n=c+3c_n\lambda_n.
\end{split}
\]
Since (ii) implies that $c_n\lambda_n\leq 1$, it follows that, 
for all $n\geq 3$
\[\int \lambda_n   x^{\lambda_n} \, d\mu\leq c+3,\]
and therefore, for all $n$, there exists a positive constant $C>0$
such that \[\int \lambda_n   x^{\lambda_n} \, d\mu\leq C.\]
Now, if we take an arbitrary function $f$ of the form  
$\sum_{n} b_n \lambda_n\xln\in M^1_\Lambda$, we obtain
\[
\|f\|_{L^1(\mu)}\le \sum_n |b_n|\ \|\lambda_n\xln \|_{L^1(\mu)} \le
C\sum_n |b_n| \le C' \|f\|_1 ,
\]
(the last inequality follows from \eqref{eq:lacunary}). Therefore $\mu$ is $\Lambda$-embedding.

As for  $\Lambda'$, we define $\lambda'_n=x_{a_n}$. Then 
\[
\begin{split}
\int \lambda'_n x^{\lambda'_n} \, d\mu 
&=  \sum_k c_k \lambda'_n a_k^{\lambda'_n}\ge c_n \lambda'_n a_n^{\lambda'_n}.
\end{split}
\]
Since the maximum value of $\lambda\to\lambda a^\lambda$ is greater
than $x_a/2$,  and since $x_{a_n}$ is of the same order as
$\frac{1}{1-a_n}$ when $n$ tends to $\infty$, there exists a positive
constant $C_1$ such that 
\[\int \lambda'_n x^{\lambda'_n} \, d\mu \geq C_1\frac{c_n}{1-a_n}.\]
Then, by definition of $c_n$, it follows that 
\[\int \lambda'_n x^{\lambda'_n} \, d\mu \geq C_1n.\]
Since $\| \lambda'_n x^{\lambda'_n}\|_{L^1(m)}\le 1$, $\mu$ is not $\Lambda'$-embedding.

\end{example}

The next example shows that for sequences that are not quasilacunary, the sublinearity of the measure $\mu$ is not in general a sufficient condition for $\Lambda$-embedding.

\begin{example}
Note first that if $\delta_t$ is the Dirac measure in $t\in [0,1)$, then $\|\delta_t\|_S=\frac{1}{1-t}$; therefore $\delta'_t=(1-t)\delta_t$ is a sublinear measure of unit (sublinear) norm.

Let us consider, for all positive integers $p,q$, the function $h_{p,q}(x)=x^p(1-x)^q$. We are interested in some estimates related to this function when $p,q\to\infty$ and $q/p\to 0$.

First, note that the maximum of $(1-x)h_{p,q}(x)$ is attained in $t_{p,q}=\frac{p}{p+q+1}$. If $p,q\to\infty$ and $q/p\to 0$, then this maximum value is of order $\frac{1}{p^{q+1}}$.

Secondly, to estimate $\|h_{p,q}\|_1$, note that it is equal to $B(p+1,q+1)$ (Euler's beta function), and thus (by standard estimates for $B$) we have
\[
\|h_{p,q}\|_1\sim \frac{(p+1)^{p+\frac{1}{2}}(q+1)^{q+\frac{1}{2}}}{(p+q+2)^{p+q+\frac{3}{2}}  }.
\] 
After some simple computations using the definition of $e$, it follows that
\begin{equation}\label{eq:hpq}
\frac{\int |h_{p,q}| d\delta'_{t_{p,q}} }{ \|h_{p,q}\|_1 } \sim (q+1)^{\frac{1}{2}}.
\end{equation}

Define then the sequence $\Lambda$ by adding together the groups of consecutive integers from $k^7$ to $k^7+k^5$, for all $k\ge 1$. Then
\[
\sum_n \frac{1}{\lambda_n}= \sum_{k=1}^\infty \sum_{j=0}^{k^5} \frac{1}{k^7+j}\le \sum_{k=1}^\infty \frac{k^5+1}{k^7}<\infty.
\] 
Since $h_{p,q}$ is a polynomial in $x^p, x^{p+1},\dots, x^{p+q}$, we have $h_{k^7,k^5}\in \Mtz$ for all $k\ge 1$. Define the measure $\mu=\sum_{k\ge 1}\frac{1}{k^2}\delta'_{t_{k^7,k^5}}$. Then $\mu$ is sublinear, with $\|\mu\|_S\le \frac{\pi^2}{6}$. 

On the other hand,
by~\eqref{eq:hpq} we have
\[
\int|h_{k^7,k^5}|d\mu\ge\frac{1}{k^2} \int |h_{k^7,k^5}| d\delta'_{t_{k^7,k^5}}
\ge\frac{1}{k^2} (k^5+1)^{\frac{1}{2}}\|h_{k^7,k^5}\|_1,
\]
and therefore
\[
\sup_k \frac{\int|h_{k^7,k^5}|d\mu}{\|h_{k^7,k^5}\|_1}\to \infty.
\]
Thus $\mu$ is not $\Lambda$-embedding.
\end{example}


\section{The sequence $\lambda_n=n^2$}

To show an example of what can be obtained going beyond quasilacunary sequences, we give in this section a more precise estimate for the function $\kappa$ in the case of the sequence $\lambda_n=n^2$. This will be done in several steps, which essentially make explicit the calculations that are involved in the general results, as found, for instance, in \cite[4.2]{be95}. The sequence $n^2$ is a main example of a \emph{standard} sequence, that is, one for which $\lambda_{n+1}/\lambda_n\to1$; it is a test for several unsolved questions in the theory of M\"untz spaces (see~\cite{gl05}). For the rest of this section, we will denote $\Lambda=(n^2)_{n\ge 1}$ and $\tilde\Lambda=(n^2+1)_{n\ge 1}$. 

First, if $f(x)=a_\gamma x^\gamma+ \sum_{n=1}^m a_n x^{\gamma_n}$, then by standard euclidian space arguments we have
\begin{equation}\label{eq:gram}
 |a_\gamma|\le d^{-1}\|f\|_2,
\end{equation}
where $d$ is the distance in $L^2[0,1]$ from $x^\gamma$ to the span of the functions $x^{\gamma_1}, \dots, x^{\gamma_m}$. 
A calculation involving biorthogonal bases and Cauchy determinants
that can be found in \cite[pp. 176--177]{be95} shows that 
\[
 d=\frac{1}{\sqrt{2\gamma+1}}\prod_{n=1}^m \left| \frac{\gamma-\gamma_n}{\gamma+\gamma_n+1}\right|.
\]

We intend now to apply~\eqref{eq:gram} to the case of the M\"untz space $M_{\tilde\Lambda}^1$. Suppose
 $f(x)=\sum_{n\ge 1} \tilde a_n x^{n^2+1}$. According to~\eqref{eq:gram}, we have
\begin{equation}\label{eq:gram2}
 |\tilde a_m|\le \sqrt{2m^2+3} \mathop{\prod_{n=1}^\infty}_{n\not=m} \left| \frac{m^2+n^2+3}{m^2-n^2}\right| \|f\|_2.
\end{equation}
We break in three parts the infinite product in the right-hand side. The first two are estimated simply:
\begin{equation}\label{eq:prod1}
 \prod_{n=1}^{m-1} \left| \frac{m^2+n^2+3}{m^2-n^2}\right| \le
 \prod_{n=1}^{m-1}  \frac{(m+n)^2}{m^2-n^2}=\frac{(2m-1)!}{m!(m-1)!}.
\end{equation}

\begin{equation}\label{eq:prod2}
 \prod_{n=m+1}^{2m} \left| \frac{m^2+n^2+3}{m^2-n^2}\right| \le
 \prod_{n=m+1}^{2m}  \frac{(m+n)^2}{n^2-m^2}=\frac{(3m)!}{(2m)!m!}.
\end{equation}

As for the third, we have
\[
 \prod_{n=2m+1}^{\infty} \left| \frac{m^2+n^2+3}{m^2-n^2}\right| =
 \prod_{n=2m+1}^{\infty} \left(1+\frac{2m^2+3}{n^2-m^2}\right).
\]
Using the inequality $\log(1+x)\le x$, we obtain
\[
\begin{split}
 \log \left( \prod_{n=2m+1}^{\infty} \left(1+\frac{2m^2+3}{n^2-m^2}\right)  \right)&=\sum_{n=2m+1}^\infty \log  \left(1+\frac{2m^2+3}{n^2-m^2}\right)\\
&\le
(2m^2+3) \sum_{n=2m+1}^\infty \frac{1}{n^2-m^2}\le \frac{(\log 3)(2m^2+3)}{2m},
\end{split}
\]
and thus
\begin{equation}\label{eq:prod3}
  \prod_{n=2m+1}^{\infty} \left| \frac{m^2+n^2+1}{m^2-n^2}\right|\le 3^{\frac{2m^2+3}{2m}}.
\end{equation}
Remembering now that Stirling's formula says that
\[
 \sqrt{2\pi N}\left(\frac{N}{e}\right)^N e^{\frac{1}{12N+1}}\le  N!
\le  \sqrt{2\pi N}\left(\frac{N}{e}\right)^N e^{\frac{1}{12N}}
\]
we obtain from~\eqref{eq:prod1},~\eqref{eq:prod2}, and~\eqref{eq:prod3}, after some majorizations,  
\[
 |\tilde a_m|\le m 3^{\frac{8m^2+3}{2m}}\|f\|_2.
\]

Consider now a polynomial $p(x)=\sum_m a_m x^{m^2}\in\Mtz$. If $\tilde p(x)=\int_0^x p(t)\, dt$, then $\tilde p(x)=\sum_m \frac{a_m}{m^2+1} x^{m^2+1}\in M_{\tilde\Lambda}^1$, and we may apply the previous estimate to $\tilde p$, obtaining
\[
 |a_m| \le m(m^2+1)3^{\frac{8m^2+3}{2m}}\|\tilde p\|_2
\le  m(m^2+1)3^{\frac{8m^2+3}{2m}}\|\tilde p\|_\infty \le 100^m \|p\|_1.
\]

Take  $\eps>0$. For any $x\in[0, 1-\eps)$, we have
\[
 |p(x)|\le \sum_m |a_m| (1-\eps)^{m^2}\le
\sum_m  100^m (1-\eps)^{m^2} \|p\|_1 \le 
  C_1 e^{\frac{C}{\scriptstyle\eps}} \|p\|_1.  
\]
The last estimate is obtained from the fact \cite[p 57]{ha71} that, for $a>1$,
\[
 \sum_{m\ge1} a^m x^{m^2} \sim e^{-\frac{(\log a)^2}{4\log x}}
\]
when $x\nearrow 1$.

In conclusion, for the sequence $\Lambda=(i^2)_{i\ge1}$ we can use in Theorem~\ref{th:generalsufficient} and Corollary~\ref{co:mu(Jeps)} the function
\[
 \kappa(t)=C_1 e^{\frac{C}{\scriptstyle 1-t}}
\]
in order to obtain sufficient conditions for a $\Lambda$-embedding measure $\mu$.
One can see that there is a big gap between these conditions and the necessary sublinearity given by Proposition~\ref{pr:sublinear}.

\section{Applications: weighted composition operators}\label{se:compo}

A starting point for this paper was Section 4.2 from \cite{aathesis}, which studies certain weighted composition operators on $\Mtz$. We show below how the main results therein fit in our general frame.  In this section $\Lambda$ is an arbitrary fixed sequence.

Let $\phi,\psi$ be two Borel functions on $[0,1]$, such that $\phi([0,1])\subset [0,1]$. For any Borel function $f$ one can define  
\[
C_\phi(f)=f\circ \phi,\qquad \T_\psi(f)=\psi f.
\]
At this level of generality, $C_\phi$ and $\T_\psi$ are linear mappings in the vector space of Borel functions. We are interested in conditions under which the \emph{weighted composition operator} $ \T_\psi\circ C_\phi$ maps $\Mtz$ boundedly into $L^1$; in case this happens, we may want to compute the essential norm.

There is a standard way to view this as an embedding problem. Recall that the \emph{pullback} of a measure $\nu$ by $\phi$ is the measure $\phi^*\nu$ on $[0,1]$ defined by
\[
\phi^*\nu(E)=\nu (\phi^{-1}(E))
\]
for any Borel set $E$. Then the formula
\begin{equation}\label{eq:comp-measure}
 \int_0^1 \psi(x) f(\phi(x))\, dx= \int_{[0,1]} f\, d(\phi^*(\psi dm))
\end{equation}
is easily checked on characteristic functions, whence it extends to all Borel functions on~$[0,1]$. 

\begin{lemma}\label{le:relation-embedding}
Define the measure $\mu= \phi^*(\psi dm)$. Then

(i)  $\T_\psi\circ C_\phi$ is bounded from $\Mtz$ to $L^1$ if and only if $\mu$ is a $\Lambda$-embedding measure.

 (ii) $\T_\psi\circ C_\phi$ is compact from $\Mtz$ to $L^1$ if and only if $\iota_\mu$ is compact.

(iii) Suppose that $m(\phi^{-1}(\{1\}))=0$. If $\mu'_n=\mu_{|[1-\frac{1}{n}, 1]}$, then
\[
\|\T_\psi\circ C_\phi\|_e=\|\iota_\mu\|_e=\lim_{n\to\infty}\|\iota_{\mu'_n}\|.
\]
\end{lemma}

\begin{proof}
Formula~\eqref{eq:comp-measure} implies that the map $J$ defined by $J(f)=\psi (f\circ \phi)$ is an isometry from $L^1(\mu)$ into $L^1$. Since we have 
\begin{equation}\label{eq:8.0}
 \T_\psi\circ C_\phi=J\circ\iota_\mu,
\end{equation}
(i) and (ii) follow immediately.

As for (iii), we will apply Lemma~\ref{le:abstractessential} to the case $X=\Mtz$, $(E,\nu)=([0,1], m)$, $T=\T_\psi\circ C_\phi$, and $E_n=\phi^{-1}([1-\frac{1}{n}, 1])$. The hypothesis on $\phi$ implies that 
\[
 m\left(\bigcap_n \phi^{-1}([1-\frac{1}{n}, 1])\right)=0;
\]
on the other side, since $P_n$ is the natural projection onto $L^1(\phi^{-1}([1-\frac{1}{n}, 1]))$, one verifies easily that 
\[
(I-P_n)T=J\circ(\iota_\mu-\iota_{\mu'_n)},
\]
whence $(I-P_n)T$ is compact by Proposition~\ref{pr:bicompact}. It follows then from Lemma~\ref{le:abstractessential} that
\[
\|\T_\psi\circ C_\phi\|_e =\lim_n\|J\circ \iota_{\mu'_n}  \|
=\lim_n\| \iota_{\mu'_n}  \|. \qedhere
\]
\end{proof}

One can now recapture some of the results in \cite{aathesis}, where a certain regularity is assumed for the functions $\phi$ and $\psi$; in this case we can describe more precisely the measure $\phi^*(\psi\,dm)$, as shown by the following lemma, whose proof is a simple computation.

\begin{lemma}\label{le:density}
Suppose $[a,b]\subset[0,1]$ has the property that $\phi^{-1}([a,b])$ is the union of $p$ intervals $[x_i,y_i]$, which have at most an end point in common, and, for each $i=1,\dots,p$, $\phi_i:=\phi_{|[x_i,y_i]}$ is a function in $C^1([x_i,y_i])$, with $\phi_i'(x)\not=0$. Then the restriction of $\phi^*(\psi\,dm)$ to $[a,b]$ is absolutely continuous, with density 
\begin{equation}\label{eq:density}
\rho(x)=\sum_{i=1}^p \frac{\psi(\phi_i^{-1}(x))}{\phi_i'(\phi_i^{-1}(x))}.
\end{equation}
\end{lemma}

The next  proposition is Lemma 4.2.8 in \cite{aathesis}. 

\begin{proposition}\label{pr:ihab1}
Suppose that $\phi\in C^1$. Then $  C_\phi$ is bounded from $\Mtz$ to $L^1$ if and only if either $\phi([0,1])\subset [0,1)$, or $\phi^{-1}(\{1\})\subset \{0,1\}$ and $\phi(x_0)=\tcr{1} \Longrightarrow \phi'(x_0)\not=0 $.
\end{proposition}

\begin{proof}
If $\phi([0,1])\subset [0,a]$ with $a<1$, then $\phi^*dm$ is supported in $[0,a]$ and thus $\Lambda$-embedding by Corollary~\ref{co:compactsupport}. \tcr{By Lemma~\ref{le:relation-embedding}, we conclude that $C_\phi$ is bounded from $\Mtz$ to $L^1$.}

Suppose now that $x_0\in [0,1]$ satisfies $\phi(x_0)=1$ and $\phi'(x_0)=0$. If $\delta>0$, let $\eps>0$ be such that $|\phi'(x)|\le \delta$ for $|x-x_0|<\eps$.  Then $\phi(x)\ge 1-\delta\eps$ for $|x-x_0|<\eps$, whence $(\phi^*m)(J_{\delta\eps})\ge \eps$. Therefore $\frac{(\phi^*m)(J_{\delta\eps})}{\delta\eps}\ge \delta^{-1}$, which implies by Lemma~\ref{le:the function} that $\phi^*m$ is not $\Lambda$-embedding \tcr{and therefore $C_\phi$ is not bounded from $\Mtz$ to~$L^1$}. 

Thus, if \tcr{$C_\phi$ is bounded from $\Mtz$ to $L^1$}, then $\phi(x_0)=1$ implies $\phi'(x_0)\not=0$. This already proves the necessity part of the proposition, since if $x_0\in (0,1)$ and $\phi(x_0)=1$ then $x_0$ is a local maximum for $\phi$ and thus $\phi'(x_0)=0$.

Finally, let us suppose that $\phi^{-1}(\tcr{\{1\}})=\{0,1\}$, $\phi'(0)\not=0$, $\phi'(1)\not=0$ (the other cases are simpler). Then there is $\delta>0$ such that $\phi^{-1}([1-\delta,1])$ is the union of two intervals $[0,a_1]$ and $[a_2,1]$ on which $\phi'$ is not zero. We may apply Lemma~\ref{le:density} to conclude that
$\phi^*m$ restricted to $J_\delta$ is absolutely continuous with respect to Lebesgue measure, with bounded density; it follows then by Corollary~\ref{co:compactsupport} (ii) that $\phi^*m$ is  $\Lambda$-embedding, \tcr{that is, $C_\phi$ is bounded}. 
\end{proof}

We will now state  a definition appearing in~\cite{aathesis}. One says that the measurable function $\phi:[0,1]\to [0,1]$ \emph{satisfies the condition $(\alpha)$} if:

(a) $\phi^{-1}(1)=\{x_1, \dots x_p\}$ is finite;

(b) there exists $\eps>0$ such that, for each $i=1,\dots, p$, $\phi\in C^1([x_i-\eps,x_i])$ and $\phi\in C^1([x_i,x_i+\eps])$;

(c) $\phi'_-(x_i)>0$ and $\phi'_+(x_i)\tcr{<}0$ for all $i$, \tcr{where $\phi'_-(x_i)$ (respectively $\phi'_+(x_i)$) is the left (respectively right) derivative of $\phi$ at $x_i$};

(d) there exists $\alpha<1$ such that, if $x\not\in\bigcup_{i=1}^p (x_i-\eps, x_i+\eps)$, then $\phi(x)\le \alpha$. 

The next result combines Theorems~4.2.11 and~4.2.15 from \cite{aathesis}.

\begin{proposition}\label{pr:ihab2}
 Suppose $\phi:[0,1]\to [0,1]$ satisfies the condition $(\alpha)$, while $\psi:[0,1]\to\CC$ is continuous. Then $\T_\psi\circ C_\phi$ is bounded, and
\[
 \|\T_\psi\circ C_\phi\|_e=\sum_{i=1}^p |\psi(x_i)| L(x_i),
\]
 where
\[
 L(x_i)=\begin{cases}
         \frac{1}{\phi'_-(x_i)} +  \frac{1}{|\phi'_+(x_i)|} &\mbox{if } x_i\in(0,1),\\
 \frac{1}{\phi'_-(x_i)}&\mbox{if } x_i=1,\\
\frac{1}{|\phi'_+(x_i)|}&\mbox{if } x_i=0.
        \end{cases}
\]
\end{proposition}

\begin{proof}
If we define $\mu=\phi^*(\psi\,dm)$, then condition $(\alpha)$ implies that,
for some $\delta>0$, $\phi^{-1}([1-\delta,1])$ is the union of $2p$ intervals $[x_i-\eta_i,x_i]$ and $[x_i,x_i+\eta'_i]$, $i=1,\dots,p$, with $\phi(x_i-\eta_i)=\phi(x_i+\eta'_i)=1-\delta$, and on each of these intervals $\phi'$ is not zero.
Thus the interval $[1-\delta,1]$ satisfies the hypothesis of Lemma~\ref{le:density}. We may then apply Corollary~\ref{co:essential-abs-cont-limit} to $\mu$, and formula~\eqref{eq:density} gives $a=\sum_{i=1}^p |\psi(x_i)| L(x_i)$. Therefore $\iota_\mu$ is bounded and its essential norm is $\sum_{i=1}^p |\psi(x_i)| L(x_i)$. Part (iii) of Lemma~\ref{le:relation-embedding} finishes then the proof.
\end{proof}

One should note that both Propositions~\ref{pr:ihab1}
and~\ref{pr:ihab2} are valid for any $\Lambda$. This is a consequence
of imposing rather strong hypotheses on $\phi$ and $\psi$
(correspondingly, on the measure~$\mu$).\\

\noindent\textbf{Acknowledgment:} The authors wish to thank Pascal Lef\`evre
for helpful discussions on M\"untz spaces, starting point of this work.


\begin{thebibliography}{1}

\bibitem{aathesis}
I.~Al~{A}lam.
\newblock {\em G{\'e}om\'etrie des espaces de M{\"u}ntz et op{\'e}rateurs de
  composition {\`a} poids}.
\newblock PhD thesis, Universit\'e {L}ille 1, 2008.

\bibitem{aa09}
I.~Al~Alam.
\newblock Essential norms of weighted composition operators on {M}\"untz
  spaces.
\newblock {\em J. Math. Anal. Appl.}, 358(2), 2009.

\bibitem{be95}
P.~Borwein and T.~Erd{\'e}lyi.
\newblock {\em Polynomials and polynomial inequalities}, volume 161 of {\em
  Graduate Texts in Mathematics}.
\newblock Springer-Verlag, New York, 1995.

\bibitem{gl05}
V.~I. Gurariy and W.~Lusky.
\newblock {\em Geometry of {M}\"untz spaces and related questions}, volume 1870
  of {\em Lecture Notes in Mathematics}.
\newblock Springer-Verlag, Berlin, 2005.

\bibitem{ha71}
G.~H. Hardy.
\newblock {\em Orders of infinity. {T}he {\it {I}nfinit\"arcalc\"ul}\ of {P}aul
  du {B}ois-{R}eymond}.
\newblock Hafner Publishing Co., New York, 1971.
\newblock Reprint of the 1910 edition, Cambridge Tracts in Mathematics and
  Mathematical Physics, No. 12.

\bibitem{sp08}
A.~Spalsbury.
\newblock Perturbations in {M}\"untz's theorem.
\newblock {\em J. Approx. Theory}, 150(1):48--68, 2008.

\end{thebibliography}
\end{document}